\documentclass[a4paper, 11pt, reqno]{amsart}

\usepackage{amsmath, amsthm, amsfonts, amssymb}
\numberwithin{equation}{section}
\usepackage{enumerate}
\usepackage[bookmarks=false]{hyperref} 
\usepackage{tikz-cd}

\usepackage{xcolor}
\usepackage{mathtools}

\usepackage{verbatim} 
\usepackage[normalem]{ulem}

\title{Weak Exactness and Amalgamated Free Product of Von Neumann Algebras}
\author{Kai Toyosawa}

\email{}

\thanks{}

\newtheorem{thm}{Theorem}[section]

\newtheorem{cor}[thm]{Corollary}
\newtheorem{lem}[thm]{Lemma}
\theoremstyle{definition}

\newtheorem{defn/lem}[thm]{Definition/Lemma}
\newtheorem{rem}[thm]{Remark}
\newtheorem{exmp}[thm]{Example}

\newcommand{\id}{\operatorname{id}}

\newcommand{\supp}{\operatorname{supp}}

\newcommand{\ds}{{\sharp\kern-.5pt\sharp}}

\newcommand\restr[2]{\ensuremath{\left.#1\right|_{#2}}}

\DeclareMathOperator{\spn}{span}

\DeclareMathOperator{\nor}{nor}

\makeatletter
\DeclareRobustCommand\frownotimes{\mathbin{\mathpalette\frown@otimes\relax}}
\newcommand{\frown@otimes}[2]{%
  \vbox{
    \ialign{##\cr
      \hidewidth$\m@th#1{}_\frown$\kern-\scriptspace\hidewidth\cr
      \noalign{\nointerlineskip\kern-1pt}
      $\m@th#1\otimes$\cr
    }%
  }%
}
\makeatother

\begin{document}

\maketitle
\begin{abstract}
  We show that the amalgamated free product of weakly exact von Neumann algebras is weakly exact. This is done by using a universal property of Toeplitz-Pimsner algebras and a locally convex topology on bimodules of von Neumann algebras, which is used to characterize weakly exact von Neumann algebras.

\end{abstract}


\section{Introduction}

Weakly exact von Neumann algebras were introduced by Kirchberg \cite{MR1403994} in parallel with exact C$^{*}$-algebras and have since became an effective tool for exhibiting structural results in von Neumann algebras. Ozawa showed in \cite{MR2370001} that for a countable discrete group $\Gamma$, the associated group von Neumann algebra $L\Gamma$ is weakly exact if and only if $\Gamma$ is an exact group. In the seminal paper \cite{MR2079600}, Ozawa proved that the group von Neumann algebra $L\Gamma$ of a hyperbolic group $\Gamma$ is solid, that is, for any diffuse von Neumann subalgebra $A$ in $L\Gamma$, the relative commutant $A' \cap L\Gamma$ is injective. As a consequence, every nonamenable subfactor of a free group factor is prime. Ozawa's proof is generalized in \cite[Theorem 15.1.5]{MR2391387} to the group von Neumann algebras associated to a large class of exact groups called biexact groups \cite[Definition 15.1.2]{MR2391387}, for which many other structural results are known, see for example \cite{MR2052608,MR3087388,MR3259044,MR3494487,MR4484235}.
Ding and Peterson further extended the notion of biexactness for general weakly exact von Neumann algebras in \cite{ding2023biexact}, and showed that a generalized solidity property holds for biexact von Neumann algebras.

It is well known that exactness of C$^{*}$-algebras is preserved under subalgebras, increasing unions, quotients (\cite{MR1322641} and \cite{MR1403994}), and amalgmated free products (\cite{MR2039095} and \cite{MR1880402}). The proofs for permanence properties of exactness in the C$^{*}$-algebra setting rarely adapt directly to weakly exact von Neumann algebras, in part due to the fact that the ultraweak topology is much weaker than the norm topology, often making it difficult to show the properties that hold for a norm-closed set also hold for its ultraweak-closure. Another obstacle comes from the fact that weak exactness only passes to von Neumann subalgebras with conditional expectation, and not to all subalgebras in general.

In \cite{MR3004955}, Isono extended the notion of weakly exact von Neumann algebras to dense C$^{*}$ subalgebras. He showed that a von Neumann algebra $M$ with a separable predual is weakly exact in the sense of Kirchberg if there exists an ultraweakly dense C$^{*}$-subalgebra $A$ that is weakly exact in $M$. This enabled him to prove that for von Neumann algebras with separable predual, weak exactness is preserved under tensor products (also shown in \cite[Corollary 14.2.5]{MR2391387}), crossed products with exact groups, and increasing unions. Using these tools, Isono and Houdayer later showed a spectral gap rigidity result inside crossed product von Neumann algebras arising from arbitrary actions of biexact discrete groups on amenable $\sigma$-finite von Neumann algebras \cite{MR3555359}. The question of whether taking free products preserves weak exactness was left open.

Ding, Kunnawalkam Elayavalli, and Peterson in \cite{MR4675043}, and Ding and Peterson in \cite{ding2023biexact} utilized a certain locally convex topology on bimodules of a von Neumann algebra in order to study the notions of proper proximality \cite{MR4258166} and biexactness in the setting of von Neumann algebras. If $M$ is a von Neumann algebra and $X\subset \mathbb{B}(\mathcal{H})$ is an operator $M$-bimodule, then there exists a locally convex topology on $X$, called the $M$-topolgy, that lies between the uniform and the ultraweak topologies. This topology was first introduced by Magajna in \cite{MR1616512} and \cite{MR1750836}, and later related to Ozawa's topology in \cite{MR2730894}. It is proved in \cite{ding2023biexact} that a von Neumann algebra $M\subset \mathbb{B}(\mathcal{H})$ is weakly exact if and only if the inclusion $M \subset\mathbb{B}(\mathcal{H})$ is $M$-nuclear, meaning that there exist nets of u.c.p.\ (unital completely positive) maps $\phi_{i}\colon M\to \mathbb{M}_{n(i)}(\mathbb{C})$ and $\psi_{i}\colon \mathbb{M}_{n(i)}(\mathbb{C}) \to \mathbb{B}(\mathcal{H})$ such that $\psi_{i}\circ \phi_{i}(x)\to x$ in the $M$-topology on $\mathbb{B}(\mathcal{H})$ for every $x\in M$. Adapting results from 
\cite{MR2211141} to this characterization of weak exactness, they showed that when $M_{1}$ and $M_{2}$ are weakly exact von Neumann algebras with nondegenerate normal states $\omega_{1}$ and $\omega_{2}$, respectively, then the free product von Neumann algebra $M = (M_{1},\omega_{1})* (M_{2},\omega_{2})$ is weakly exact.

In this paper, we generalize the above result and show that weak exactness is preserved under taking amalagamated free products of von Neumann algebras: 
\begin{thm} \label{thm: amalg free prod preserves wk exact}
  Let $M_{i}$, $i\in I$, be weakly exact $\sigma$-finite von Neumann algebras such that each $M_{i}$ contains a copy of a fixed von Neumann subalgebra $(B,\varphi)$ with a faithful normal state $\varphi$. Assume each $M_{i}$ admits a faithful normal conditional expectation $E_{i}\colon M_{i}\to B$. Then the amalgamted free product von Neumann algebra $\overline{*}_{B}(M_{i},E_{i})_{i\in I}$ is weakly exact.
\end{thm}

The proof of the above theorem is inspired by \cite{MR1880402} and \cite[Section 4.8]{MR2391387}. In \cite{MR1880402}, Dykema and Shlyakhtenko showed that the reduced amalgamated free products of exact C$^{*}$-algebras are quotients of canonically associated Cuntz-Pimsner algebras, which are exact, so the reduced amalgamated free products themselves are exact as well. A similar proof was presented in \cite{MR2391387}, using Toeplitz-Pimsner algebras instead of Cuntz-Pimsner algebras. We will follow a similar line of reasoning for our proof. The major diffculty in proving Theorem \ref{thm: amalg free prod preserves wk exact} lies in the facth that the techniques of \cite{MR1880402} and \cite{MR2391387} cannot be used when the norm topology is replaced by the $M$-topology, as there is no natural ultraweak topology on the associated Toeplitz-Pimsner algebra that is compatible with the ultraweak topology on $M$. In section 2.4 we show, however, that a universal property of the Toeplitz-Pimsner algebra can be adapted to the setting of the $M$-topology (Lemma \ref{lem: normal hom ext hom on toep}), and we will use that to prove Theorem \ref{thm: amalg free prod preserves wk exact}.

Inspired by the HNN extensions of groups, Ueda introduced the HNN extension constructions for von Neumann algebras in \cite{MR2152505} and \cite{MR2546003}. By \cite[Proposition 3.1]{MR2546003}, the HNN extension of a $\sigma$-finite von Neumann algebra $M$ can be naturally identified with a corner of an amalgamted free product, which is weakly exact when the original von Neumann algebra $M$ is weakly exact. Therefore, we obtain as a corollary of Theorem \ref{thm: amalg free prod preserves wk exact} that weak exactness is preserved under HNN extensions:

\begin{cor} \label{cor: wk exact for HNN}
  Let $B\subset M$ be $\sigma$-finite von Neumann algebras and $\theta\colon B\to M$ a faithful normal $*$-homomorphism. Assume there exist faithful normal conditional expectations $E^{M}_{B}\colon M\to B$ and $E^{M}_{\theta(B)}\colon M\to \theta(B)$. If $M$ is weakly exact, then the HNN extension of $M$ by $\theta$ is also weakly exact.
\end{cor}

\textbf{Acknowledgements.} The author thanks Jesse Peterson for suggesting Lemma \ref{lem: p_nor invariant}, along with numerous valuable conversations throughout the project. The author also thanks Ishan Ishan and Srivatsav Kunnawalkam Elayavalli for their helpful comments.

\section{Preliminaries}

In this paper, when $\mathcal{K}$ is a Hilbert space, $\mathbb{B}(\mathcal{K})$ denotes the space of bounded linear operators on $\mathcal{K}$. When $A,B$ are C$^{*}$-algebras, we denote by $A\odot B$ their algebraic tensor product and $A\otimes B$ their minimal tensor product. 

\subsection{Amalgamated free product}

We briefly recall the construction of amalgamated free products of C$^{*}$-algebras and von Neumann algebras. We refer the reader to \cite{MR1217253} for more detailed discussions on C$^{*}$-algebras case, \cite{MR1198815} for tracial von Neumann algebras, and \cite{MR1738186} for $\sigma$-finite von Neumann algebras.

Let $A_{i}$, $i\in I$, be unital C$^{*}$-algebras such that each contains a copy of a fixed unital C$^{*}$-subalgebra $D$. We assume that for each $i\in I$, there exists a faithful conditional expectation $E_{i}\colon A_{i}\to D$. Let $\langle\,,\,\rangle_{i}$ be the $D$-valued inner product on $A_{i}$ given by $\langle a,b\rangle_{i} = E_{i}(a^{*}b)$ for $a,b\in A_{i}$, and denote by $L^{2}A_{i}$ the natural C$^{*}$-correspondence over $D$ obtained from $A_{i}$ by separation and completion. Let $\xi_{i} \in L^{2}A_{i}$ be the distinguished element arising from $1\in A_{i}$ and $L^{2}A_{i}^{o} = L^{2}A_{i}\ominus \xi_{i}D$ be the complementing $D$-submodule of $\xi_{i}D$ in $L^{2}A_{i}$. We define the free product Hilbert $D$-module $(\mathcal{H},\xi) = *_{D}(L^{2}A_{i},\xi_{i})$ by
\[\mathcal{F} = \xi D \oplus \bigoplus_{n\geqslant 1}\bigoplus_{i_{1}\neq\cdots \neq i_{n}}L^{2}A_{i_{1}}^{o}\otimes_{D}\cdots \otimes_{D} L^{2}A_{i_{n}}^{o},\]
where $\xi D$ is the trivial Hilbert $D$-module with distinguished element $\xi = 1_{D}$, and $\otimes_{D}$ denotes the interior tensor product of C$^{*}$-correspondences over $D$. For each $i\in I$, we let
\[\mathcal{F}(i) = \xi D \oplus \bigoplus_{n\geqslant 1}\bigoplus_{\substack{i_{1}\neq i,\\i_{1}\neq\cdots \neq i_{n}}}L^{2}A_{i_{1}}^{o}\otimes_{D}\cdots \otimes_{D} L^{2}A_{i_{n}}^{o},\]
and so we have a canonical isomorphism $\mathcal{F}\simeq L^{2}A_{i} \otimes_{D} \mathcal{F}(i)$ for each $i\in I$. Moreover, this isomorphism gives a $*$-representation $\lambda_{i}\colon A_{i}\to \mathbb{B}(\mathcal{F})$. We observe that $\restr{\lambda_{i}}{D}$ is the canonical left action of $D$ on the C$^{*}$-correspondence $\mathcal{F}$ over $D$, so $\restr{\lambda_{i}}{D} = \restr{\lambda_{j}}{D}$ for every $i,j\in I$. The \textit{(reduced) amalgamated free product C$^{*}$-algebra $(A,E) = *_{D}(A_{i}, E_{i})_{i\in I}$} is the C$^{*}$-subalgebra $A$ of $\mathbb{B}(\mathcal{F})$ generated by $\bigcup_{i\in I}\lambda_{i}(A_{i})$, together with the conditional expectation from $A$ onto $D$ given by $E(x) = \langle \xi,x\xi\rangle$.

Now suppose $M_{i}$, $i\in I$, are $\sigma$-finite von Neumann algebras having a fixed unital von Neumann subalgebra $(B,\varphi)$ with a faithful normal state $\varphi$ on $B$. Assume that $E_{i}\colon M_{i}\to B$ is a normal faithful conditional expectation and let $\varphi_{i}\ = \varphi\circ E_{i}$ for $i\in I$. Let $L^{2}B$ be the standard Hilbert space for $B$ and $\xi$ denote the implementing vector of $\varphi$. Similarly, let $L^{2}M_{i}$ be the standard Hilbert space for $M_{i}$ and $\xi_{i}$ denote the implementing vector of $\varphi\circ E_{i}$. We write $M_{i}^{o} = \ker E_{i}$ and $L^{2}M_{i}^{o} = L^{2}M_{i}\ominus L^{2}B$. We define the free product Hilbert space $\mathcal{F}$ by
\[\mathcal{F} = L^{2}B \oplus  \bigoplus_{n\geqslant 1}\bigoplus_{i_{1}\neq\cdots \neq i_{n}}L^{2}M_{i_{1}}^{o}\otimes_{B}\cdots \otimes_{B} L^{2}M_{i_{n}}^{o},\]
where $\otimes_{B}$ denotes Sauvageot’s relative tensor product \cite{MR0703809}. For each $i\in I$, we let
\[\mathcal{F}(i) = L^{2}B \oplus  \bigoplus_{n\geqslant 1}\bigoplus_{\substack{i_{1}\neq i,\\i_{1}\neq\cdots \neq i_{n}}}L^{2}M_{i_{1}}^{o}\otimes_{B}\cdots \otimes_{B} L^{2}M_{i_{n}}^{o},\]
and from the canonical isomorphism $\mathcal{F}\simeq L^{2}M_{i} \otimes_{B} \mathcal{F}(i)$ we obtain (abusing the notation) a $*$-representation $\lambda_{i}\colon A_{i}\to \mathbb{B}(\mathcal{F})$ such that $\restr{\lambda_{i}}{B}$ coincides with the canonical left action $\lambda$ of $B$ on $\mathcal{H}$ for every $i\in I$.

Let $M$ be the von Neumann algebra $\big(\bigcup_{i\in I}\lambda_{i}(M_{i})\big)''$ and $\widetilde{\varphi}$ be the faithful vector state on $M$ given by $\xi_{0}$. There exists a normal faithful conditional expectation $E\colon M\to \lambda(B)$ conditioned by $\widetilde{\varphi}$ by Takesaki's Theorem \cite{MR0303307} such that
\[E(x_{1}\cdots x_{n}) = 0\]
for $n\in \mathbb{N}$, $x_{k} \in M_{i_{k}}^{o}$ with $i_{1}\neq i_{2}\neq \cdots i_{n}$. The pair $(M,E)$ is called the \textit{amalgamated free product von Neumann algebra of $(M_{i},E_{i})_{i\in I}$ over $B$}, and is denoted by $(M,E)= \overline{*}_{B}(M_{i},E_{i})_{i\in I}$. The construction of $(M,E)$ is independent of the choice of $\varphi$. Furthermore, $L^{2}(M,\widetilde{\varphi})$ can be canonically identified with $\mathcal{F}$.

If each $M_{i}$ contains a unital C$^{*}$-algebra $A_{i}$ and a fixed unital C$^{*}$-subalgebra $D\subset A_{i}$ such that $E_{i}(A_{i}) = D$, then by the universal property of reduced amalgamted free products (see e.g., \cite[Theorem 4.7.2]{MR2391387}), the C$^{*}$-subalgebra of $(M,E)$ generated by $\bigcup_{i\in I}\lambda_{i}(A_{i})$ is canonically isomorphic to the (reduced) amalgamated free product C$^{*}$-algebra $(A, \restr{E}{A}) = *_{D}(A_{i},\restr{E_{i}}{A_{i}})_{i\in I}$, and the notations for $\lambda_{i}$ are actually compatible.

In the following sections, we will often omit writing conditional expectations $E_{i}$ in the notations for amalgamated free products and write $A = *_{D}(A_{i})_{i\in I}$ or $M = \overline{*}_{B}(M_{i})_{i\in I}$. In particular, when $I= \{1,2\}$, for brevity we will write $A = A_{1}*_{D}A_{2}$ for the (reduced) amalgamated free product of C$^{*}$-algebras $A_{1},A_{2}$ and $M = M_{1}\overline{*}_{B}M_{2}$ for the amalgamated free product of von Neumann algebras $M_{1},M_{2}$.

\subsection{A relative topology on C$^{*}$-bimodules and weak exactness}

In this section we collect some basic facts of a locally convex topology relative to von Neumann algebras on C$^{*}$-bimodules introduced by Magajna in \cite{MR1616512} and \cite{MR1750836}, and later adapted and generalized in \cite{MR4675043} and \cite{ding2023biexact}. This topology has been used to characterize weakly exact von Neumann algebras \cite[Theorem 5.1]{ding2023biexact}. We refer reader to \cite{MR4675043} and \cite{ding2023biexact} for details on this topology.

Throughout this section, we assume that $M$ is a von Neumann algebra and and $A$ is a unital, ultraweakly dense C$^{*}$-subalgebra of $M$. A \textit{(concrete) operator $A$-system} $X$ consists of a concrete embedding $X\subset \mathbb{B}(\mathcal{H})$ (here $\mathcal{H}$ is a Hilbert space) and a faithful non-degenerate representation $\pi\colon A\to \mathbb{B}(\mathcal{H})$ such that $X$ is a $\pi(A)$-bimodule. We call such embedding of $X\subset \mathbb{B}(\mathcal{H})$ together with $\pi$ \textit{a concrete realization of $X$ as an $A$-system}. We say the operator $A$-system $X$ is \textit{$(A\subset M)$-normal} if the concrete realization can be made so that $\pi$ extends to a normal representation of $M$. If $X$ is moreover a unital C$^{*}$-algebra, we will say that it is an \textit{operator $A$-C$^{*}$-algebra} or an \textit{$(A\subset M)$-normal} $A$-C$^{*}$-algebra if $X$ is so as an operator system.

We let $A^{\sharp}$ denote the subspace of $A^{*}$ consisting of linear functionals that can be extended to normal linear functionals on $M$. Given positive linear functionals $\omega,\rho \in A^{\sharp}$, we consider the seminorm on $X$ as in \cite[Section 3]{ding2023biexact} given by
\[s^{\rho}_{\omega}(x) = \inf\{\rho(a^{*}a)^{\frac{1}{2}}\left\Vert y \right\Vert\omega(b^{*}b)^{\frac{1}{2}} : x= a^{*}yb, a,b\in C_{n}(A), y\in \mathbb{M}_{n}(X), n\in\mathbb{N}\},\]
where $C_{n}(A)$ denotes the space of the $n\times 1$ column matrices with entries from $A$. We call the topology on $X$ induced by the seminorms $\{s^{\rho}_{\omega}:\omega,\rho \in (A^{\sharp})_{+}\}$ the $(A\subset M)$-topology on $X$ (or simply the \textit{$M$-topology} when $A=M$).

We denote by $X^{A\sharp A}$, or just by $X^{\sharp}$ if no confusion will arise, the space of linear functionals $\varphi \in X^{*}$ such that for any $x\in X$, the map $A\times A \ni (a,b)\mapsto \varphi(axb)$ extends to a separately ultraweakly continuous bilinear form on $M$. We call the $\sigma(X, X^{A\sharp A})$-topology the \textit{weak $(A\subset M)$-topology} (or simply the \textit{weak $M$-topology} when $A=M$). By \cite[Proposition 3.3]{ding2023biexact}, a functional $\varphi\in X^{*}$ is continuous in the $(A\subset M)$-topology if and only if $\varphi\in X^{A\sharp A}$.

We remark that every operator system $X$ can be regarded as an operator $\mathbb{C}$-system via the embedding $\mathbb{C} = \mathbb{C}1 \subset X$. In this case, the $\mathbb{C}$-toplogy on $X$ is the norm topology and the weak $\mathbb{C}$-topology is the weak topology $\sigma(X,X^{*})$.

We will need the following lemma adapted from \cite[Lemma 3.7]{ding2023biexact}. The proof is essentially identical to the proof for \cite[Lemma 3.7]{ding2023biexact} up to replacing the weak $M$-topologies by weak $(A\subset M)$-topologies.
\begin{lem} \label{lem: ext cont weak top}
	Let $M$ and $N$ be von Neumann algebras, $A\subset M$ and $B\subset N$ be unital ultraweakly dense C$^{*}$-algebras, and $E$ and $F$ be operator $A$ and $B$-systems, respectively. If $\phi\colon E\to F$ is a completely positive map such that the restriction of $\phi$ to $A$ defines a map from $A$ to $B$ that extends to an ultraweakly continuous map from $M$ to $N$, then $\phi$ is a continuous map from $E$ with the weak $(A\subset M)$-topology, to $F$ with the weak $(B\subset N)$-topology.
\end{lem}

Now, suppose $B$ is a unital $(A\subset M)$-normal $A$-C$^{*}$-algebra. We let $p_{\nor}\in A^{**}\subset B^{**}$ denote the projection corresponding to the support of the identity representation $A\to M$. As explained in \cite[Section 3.1]{ding2023biexact}, we may identify $(p_{\nor}B^{**}p_{\nor})_{*} \simeq B^{\sharp}$ by considering the restriction map to $B$, so the dual map naturally gives $B^{\sharp *}$ a von Neumann algebra structure so that 
\[B^{\sharp *}\simeq p_{\nor}B^{**}p_{\nor}\]
as von Neumann algebras. Since $p_{\nor}$ commutes with $A\subset B^{**}$, the above isomorphism preserves the natural $A$-bimodule structures on $B^{\sharp *}$ and $p_{\nor}B^{**}p_{\nor}$, so we view $M = A^{\sharp *}$ as a von Neumann subalgebra of $B^{\sharp *}$. 

If $i_{B}\colon B\to B^{\sharp *}$ is the canonical inclusion map of $B$, then $i_{B}$ is an $A$-bimodular complete order isomorphism, but it is not a $*$-homomorphism in general because $p_{\nor}$ might not be central in $B^{**}$.

We now recall a definition of weak exactness of ultraweakly dense C$^{*}$-algebras, \cite[Definition 3.1.1]{MR3004955}, which generalizes the original definition of weakly exact von Neumann algebra in \cite{MR1403994}. If $M$ is a von Neumann algebra and $A\subset M$ is a (unital) ultraweakly dense C$^{*}$-algebra, then we say \textit{$A$ is weakly exact in $M$} if for any unital C$^{*}$-algebra $B$ with a closed two-sided ideal $J$ and any representation $\pi\colon A\otimes B\to \mathbb{B}(\mathcal{K})$ with $A\otimes J\subset \ker\pi$ such that $\restr{\pi}{A\otimes \mathbb{C}}$ extends to an ultraweakly continuous representation of $M$, the induced representation $\widetilde{\pi}\colon A\odot B/J\to \mathbb{B}(\mathcal{K})$ is min-continuous. 

In particular, when $A=M$, $M$ is weakly exact in $M$ if and only if $M$ is weakly exact in the sense of Kirchberg \cite{MR1403994}. Also, if $M = A^{**}$, then $A$ is weakly exact in $A^{**}$ if and only if $A$ is an exact C$^{*}$-algebra.

Recall that if $E$ is an operator system or a C$^{*}$-algebra and if $F$ is an operator $A$-system or $A$-C$^{*}$-algebra for a (unital) ultraweakly dense C$^{*}$-subalgebra $A$ of a von Neumann algebra $M$, then as defined in \cite[Section 4.1]{ding2023biexact}, we say a c.c.p.\ (completely contractive positive) map $\phi\colon E\to F$ is \textit{$(A\subset M)$-nuclear} if there exist nets of c.c.p.\ maps $\phi_{i}\colon E\to \mathbb{M}_{n(i)}(\mathbb{C})$ and $\psi\colon \mathbb{M}_{n(i)}(\mathbb{C})\to F$ such that $\psi_{i}\circ \phi_{i}(x)$ converges to $\phi(x)$ in the $(A\subset M)$-topology for every $x\in E$. By \cite[Theorem 5.1]{ding2023biexact}, one can characterize weak exactness as follows:  if $M\subset \mathbb{B}(\mathcal{H})$ is a von Neumann algebra and $A\subset M$ is a (unital) ultraweakly dense C$^{*}$-algebra, then $A$ is weakly exact in $M$ if and only if the inclusion map $A\subset \mathbb{B}(\mathcal{H})$ is $(A\subset M)$-nuclear.

We say that a Hilbert C$^{*}$-correspondence $\mathcal{H}$ over $A$ is \textit{$(A\subset M)$-normal} if for any $\xi,\eta\in \mathcal{H}$ and $f\in A^{\sharp}$, the map $A\ni a\mapsto f(\langle \xi,a\eta\rangle)$ extends to a normal linear functional on $M$. Our main example of an $(A\subset M)$-normal Hilbert $A$-correspondence in this paper is the following:
\begin{exmp} \label{exmp: correspondence}
  Let $M$ be a von Neumann algebra and $A\subset M$ be a (unital) ultraweakly dense C$^{*}$-subalgebra of $M$. Let $\phi\colon M\to M$ be a normal u.c.p.\ map such that $\phi(A) \subset A$. Let $\mathcal{H}^{\phi}_{A}$ denote the C$^{*}$-correspondence over $A$ obtained from the algebraic tensor product $A\odot A$ by separation and completion.  Equip $\mathcal{H}^{\phi}_{A}$ with the left and right $A$-actions $a\cdot (x\otimes y)\cdot b = ax\otimes yb$ for $a,b,x,y\in A$, and the $A$-valued semi-inner product 
  \[\langle a_{1}\otimes b_{1}, a_{2}\otimes b_{2}\rangle = b_{1}^{*}\phi(a_{1}^{*}a_{2})b_{2}, \quad a_{1},a_{2},b_{1},b_{2}\in A.\]
  It follows from the normality of $\phi$ that $\mathcal{H}^{\phi}_{A}$ is an $(A\subset M)$-normal C$^{*}$-correspondence over $A$.
\end{exmp}

\subsection{Toeplitz-Pimsner algebra}

We refer the readers to \cite[Section 4.6]{MR2391387} for following constructions and properties of Toeplitz-Pimsner algebras. Let $M$ be a von Neumann algebra and $A$ be an ultraweakly dense C$^{*}$-subalgebra of $M$. Let $(\mathcal{H}, \langle\,,\,\rangle_{A})$ be a C$^{*}$-correspondence over $A$, i.e., a right Hilbert $A$-module with a faithful nondegenerate $*$-representation $\pi\colon A\to \mathbb{B}(\mathcal{H})$ that gives a left action of $A$; we often omit $\pi$ and write $a\xi$ for $\pi(a)\xi$ if there is no confusion. Assume that $\mathcal{H}$ is full, meaning $\{\langle \xi,\eta\rangle_{\mathcal{H}}: \xi,\eta\in \mathcal{H}\}$ generates $A$ as a C$^{*}$-algebra. For $n\geqslant 1$, let $\mathcal{H}^{\otimes_{A}^n} = \mathcal{H}\otimes_{A} \cdots \otimes_{A}\mathcal{H}$ denote the $n$-fold interior tensor product of $\mathcal{H}$ over $A$, and let $\mathcal{F}(\mathcal{H}) = A\oplus \bigoplus_{n\geqslant 1}\mathcal{H}^{\otimes_{A}^n}$ be the full Fock space over $\mathcal{H}$. Naturally $\mathcal{F}$ is a C$^{*}$-correspondence over $A$. Let $\pi_{\mathcal{F}(\mathcal{H})}$ denote the $*$-representation of $A$ coming from the left action on $\mathcal{F}(\mathcal{H})$. For $\xi\in \mathcal{H}$, the creation operators $\tau(\xi)\in \mathbb{B}(\mathcal{F}(\mathcal{H}))$ are defined by
\begin{equation*}
  \begin{aligned}
    \tau(\xi)(\xi_{1}\otimes\cdots\otimes\xi_{n}) &=\xi\otimes \xi_{1}\otimes\cdots\otimes\xi_{n}\,,  \qquad\xi_{1},\cdots,\xi_{n}\in \mathcal{H},\\
    \tau(\xi)(a) &= \xi a\,,  \qquad \qquad \qquad\qquad a\in A,
  \end{aligned}
\end{equation*}
and they satisfy the relations
  \begin{align}
    a\tau(\xi)b &= \tau(a\xi b)\,,  \qquad\xi \in \mathcal{H}, a,b\in A, \label{eqn: rel for toep, 1}\\
    \tau(\xi)^{*}\tau(\eta) &= \langle \xi,\eta\rangle_{A}\,, \qquad \xi,\eta\in \mathcal{H} \label{eqn: rel for toep, 2}.
  \end{align}
We define the (augmented) Toeplitz-Pimsner algebra $\mathcal{T}(\mathcal{H})$ to be the C$^{*}$-subalgebra of $\mathbb{B}(\mathcal{F}(\mathcal{H}))$ generated by $\pi_{\mathcal{F}(\mathcal{H})}(A)$ and $\{\tau(\xi):\xi\in\mathcal{H}\}$. Note that there is an action $\gamma$ (called the gauge action) of the unit circle $\mathbb{T}=\{z\in \mathbb{C}:|z|=1\}$ on $\mathcal{T}(\mathcal{H})$ such that for every $z\in \mathbb{T}$,
\begin{equation*}
  \begin{aligned}
    \gamma_{z} a&= a\,,  \qquad\quad\,\,\, a\in A,\\
    \gamma_{z}(\tau(\xi))&= z\tau(\xi)\,, \qquad \xi\in \mathcal{H}.
  \end{aligned}
\end{equation*}
A representation of $\mathcal{H}$ on a C$^{*}$-algebra $B$ is a pair $(\pi,\theta)$, with $\pi\colon A\to B$ a $*$-homomorphism and $\theta\colon \mathcal{H}\to B$ a linear map such that
  \begin{align}
    \pi(a)\theta(\xi)\pi(b) &= \theta(a\xi b)\,,  \qquad\quad\  \xi \in \mathcal{H}, a,b\in A,   \label{eqn: defn of rep of C* corresp, 1}\\
    \theta(\xi)^{*}\theta(\eta) &= \pi(\langle \xi,\eta\rangle_{A})\,, \qquad\xi,\eta\in \mathcal{H}. \label{eqn: defn of rep of C* corresp, 2}
  \end{align}

We recall that the canonical representation $(\pi_{\mathcal{F}(\mathcal{H})},\tau)$ of $\mathcal{H}$ on $\mathcal{T}(\mathcal{H})$ is universal in the sense that for any other representation $(\pi,\theta)$ of $\mathcal{H}$ there is a continuous surjective $*$-homomorphism from $\mathcal{T}(\mathcal{H})$ onto the $C^{*}(\pi(A),\theta(\mathcal{H}))$, sending $\pi_{\mathcal{F}(\mathcal{H})}(a)$ to $\pi(a)$ and $\tau(\xi)$ to $\theta(\xi)$.

We also recall the gauge-invariant uniqueness theorem for Toeplitz-Pimsner algebras (see e.g., \cite[Theorem 4.6.18]{MR2391387} or \cite{MR1986889}), which states that if $(\pi,\theta)$ is a representation of $\mathcal{H}$ such that $\pi$ is faithful, $(\pi,\theta)$ admits a gauge action, and $\pi(A)\cap \overline{\spn}\{\tau(\xi)\tau(\eta)^{*}:\xi,\eta\in\mathcal{H}\} = \{0\}$, then the representation $(\pi,\theta)$ is universal. In particular, the C$^{*}$-algebra generated by $\pi(A)$ and $\tau(\mathcal{H})$ is canonically isomorphic to $\mathcal{T}(\mathcal{H})$.

In general, even if $A$ is an ultraweakly dense C$^{*}$-subalgebra of $M$, the associated Toeplitz-Pimsner algebra $\mathcal{T}(\mathcal{H})$ needs not be an $(A\subset M)$-normal $A$-C$^{*}$-algebra. However, the following lemma states that whenever the C$^{*}$-correspondence $\mathcal{H}$ over $A$ is $(A\subset M)$-normal, then the image subspace of the support projection $p_{\nor}$ of the identity representation $A\to M$ is invariant under $\tau(\xi)\in\mathcal{T}(\mathcal{H})$ for all $\xi \in \mathcal{H}$.
\begin{lem} \label{lem: p_nor invariant}
  Let $M$ be a von Neumann algebra and $A$ be a unital ultraweakly dense C$^{*}$-subalgebra of $M$. Suppose $\mathcal{H}$ is an $(A\subset M)$-normal $C^{*}$-correspondence over $A$, and $B$ is a unital $(A\subset M)$-normal C$^{*}$-algebra. Let $p_{\nor} \in B^{**}$ denote the projection corresponding to the support of the identity representation $A\to M$.

  If $\pi\colon \mathcal{T}(\mathcal{H})\to B$ is a $*$-homomorphism that restricts to the canonical identification map on $A$, then for any $\xi \in\mathcal{H}$, we have
  \[p_{\nor}\pi(\tau(\xi))p_{\nor} = \pi(\tau(\xi))p_{\nor} \qquad \text{ in }B^{**}.\]
\end{lem}
\begin{proof}
  We fix a faithful $*$-representation $B^{**}\subset \mathbb{B}(\mathcal{K})$ for some Hilbert space $\mathcal{K}$. If $\xi\in \mathcal{H}$, then for any vector $v \in p_{\nor}\mathcal{K}$, the map
  \[A\ni a \mapsto \langle \pi(\tau(\xi))v, a\pi(\tau(\xi))v \rangle = \langle v, \pi(\tau(\xi))^{*}a\pi(\tau(\xi)) v\rangle = \langle v,\langle \xi,a\xi\rangle_{A} v\rangle\]
  extends to a normal functional on $M$ since $\mathcal{H}$ is $(A\subset M)$-normal. Hence $\pi(\tau(\xi))v \in p_{\nor}\mathcal{K}$ whenever $v\in p_{\nor}\mathcal{K}$.
\end{proof}

We recall that if $\mathcal{H}$ is a C$^{*}$-correspondence over $A$ and $B$ is a C$^{*}$-algebra, then the exterior product of $\mathcal{H}$ and $B$, denoted as $\mathcal{H}\otimes B$, is the C$^{*}$-correspondence over $A\otimes B$ obtained by separation and completion of the algebraic tensor product $\mathcal{H}\otimes B$ with respect to the $A\otimes B$-valued semi-inner product
\[\langle \xi_{1}\otimes b_{1}, \xi_{2}\otimes b_{2}\rangle_{A\otimes B} = \langle \xi_{1},\xi_{2}\rangle_{A} \otimes b_{1}^{*}b_{2}\in A\otimes B\,, \qquad \xi_{1},\xi_{2}\in \mathcal{H}, \, b_{1},b_{2}\in B.\]
We have a natural action of $\mathbb{B}(\mathcal{H})\otimes B$ on $\mathcal{H}\otimes B$ is given by
\[(x\otimes b)(\xi \otimes b') = x\xi\otimes bb'\,, \qquad x\in \mathbb{B}(\mathcal{H}), \xi \in \mathcal{H}\,, b,b'\in B.\]

\begin{lem} \label{lem: normal hom ext hom on toep}
  Suppose $M$ is a von Neumann algebra and $A$ is a unital, ultraweakly dense C$^{*}$-subalgebra of $M$. Let $\mathcal{H}$ be a C$^{*}$-correspondence over $A$ and let $B$ be a unital $(A\subset M)$-normal $A$-C$^{*}$-algebra. 

  If $\phi\colon \mathcal{T}(\mathcal{H})\to B$ is a $*$-homomorphism whose restriction to $A$, $\restr{\phi}{A}\colon A\to B$, is the canonical inclusion and is $(A\subset M)$-nuclear, then there exists a weakly nuclear $*$-homomorphism $\pi\colon \mathcal{T}(\mathcal{H})\to B^{\sharp *}= (B^{A\sharp A})^{*}$ such that $\pi(x) = (i_{B}\circ \phi)(x)$ for $x\in A\cup \tau(\mathcal{H})\cup \tau(\mathcal{H})^{*}$.
\end{lem}

\begin{proof}
  Let $p_{\nor} \in B^{**}$ denote the support projection of the inclusion $A\subset M$. We naturally identify $B^{\sharp *} \simeq p_{\nor}B^{**}p_{\nor}$ as a von Neumann algebra, and $M\subset B^{\sharp *}$ as von Neumann subalgebra. Fix a Hilbert space $\mathcal{K}$ such that $B^{\sharp *} $ is normally and faithfully represented in $\mathbb{B}(\mathcal{K})$. Since the restriction map $\restr{\phi}{A}\colon A\to B$ is $(A\subset M)$-nuclear, by \cite[Lemma 4.4]{ding2023biexact}, the c.c.p.\ map $\widetilde{\phi}_{A}\coloneqq i_{B}\circ \restr{\phi}{A}\colon A\to B^{\sharp *} \subseteq \mathbb{B}(\mathcal{K})$ gives a weakly nuclear $*$-homomorphism. By \cite{MR0512362}, if we let $(B^{\sharp *})'$ denote the commutant of $B^{\sharp *}$ in $\mathbb{B}(\mathcal{K})$, then the product map
  \[\widetilde{\phi}_{A}\times \id_{(B^{\sharp *})'} \colon A\odot (B^{\sharp *})' \ni \sum_{k}a_{k}\otimes b_{k} \mapsto \sum_{k}\widetilde{\phi}_{A}(a_{k})b_{k} \in \mathbb{B}(\mathcal{K}), \quad a_{k}\in A,\, b_{k}\in  (B^{\sharp *})'\]
  is continuous with respect to the minimum tensor norm. Let $\rho$ be the continuous extension of $\widetilde{\phi}_{A}\times \id_{(B^{\sharp *})'}$ to $A\otimes (B^{\sharp *})'$.

  Let $\mathcal{H}\otimes (B^{\sharp *})'$ denote the exterior tensor product of $\mathcal{H}$ and $(B^{\sharp *})'$, and note that it is a C$^{*}$-correspondence over $A\otimes(B^{\sharp *})'$. Define the c.c.p.\ map $\widetilde{\phi}\coloneqq i_{B}\circ \phi\colon \mathcal{T}(\mathcal{H}) \to B^{\sharp *} \subseteq \mathbb{B}(\mathcal{K})$, and consider the map $\widetilde{\tau}$ given by
  \[\widetilde{\tau}\colon \mathcal{H}\odot (B^{\sharp *})'\ni \sum_{j}\xi_{j}\otimes b_{j}\mapsto \sum_{j} \widetilde{\phi}(\tau(\xi_{j})) b_{j} \in \mathbb{B}(\mathcal{K}).\]
  Since $\phi$ is a $*$-homomorphism, we notice that for any $\xi_{1},\cdots,\xi_{n}, \eta_{1},\cdots,\eta_{m}\in\mathcal{H}$ and $b_{1},\cdots,b_{n},c_{1},\cdots,c_{m} \in (B^{\sharp *})'$, by Lemma \ref{lem: p_nor invariant} and relation (\ref{eqn: rel for toep, 2}), 

  \begin{equation} \label{eqn: ext prod with comm}
  \begin{aligned}
    \bigg(\sum_{j=1}^{n}\widetilde{\phi}(\tau(\xi_{j})) b_{j}\bigg)^{*}\bigg(\sum_{k=1}^{m}\widetilde{\phi}(\tau(\eta_{k})) c_{k}\bigg) &= \bigg(\sum_{j=1}^{n}p_{\nor}\phi(\tau(\xi_{k}))p_{\nor} b_{k}\bigg)^{*}\bigg(\sum_{k=1}^{m}p_{\nor}\phi(\tau(\eta_{k}))p_{\nor} c_{k}\bigg)\\
    &= \sum_{j,k}p_{\nor}\phi((\tau(\xi_{j}))^{*}p_{\nor}\phi(\tau(\eta_{k}))p_{\nor} b_{j}^{*}c_{k}\\
    &= \sum_{j,k}p_{\nor}\phi((\tau(\xi_{j}))^{*}\phi(\tau(\eta_{k}))p_{\nor} b_{j}^{*}c_{k}\\
    &= \sum_{j,k}p_{\nor}\phi(\tau(\xi_{j})^{*}\tau(\eta_{k}))p_{\nor} b_{j}^{*}c_{k}\\
    &= \sum_{j,k}p_{\nor}\phi(\langle \xi_{j},\eta_{k}\rangle_{A})p_{\nor} b_{j}^{*}c_{k}\\
    &= \sum_{j,k}\widetilde{\phi}((\langle \xi_{j},\eta_{k}\rangle_{A}) b_{j}^{*}c_{k}\\
    &= \sum_{j,k}(\widetilde{\phi}_{A}\times \id)(\langle \xi_{j},\eta_{k}\rangle_{A} \otimes b_{j}^{*}c_{k})\\
    &= \rho\bigg(\left\langle \sum_{j=1}^{n} \xi_{j}\otimes b_{j}, \sum_{k=1}^{m} \eta_{k}\otimes c_{k}\right\rangle_{A\otimes (B^{\sharp *})'}\bigg)\,.
  \end{aligned}
  \end{equation}
  In particular, taking $m=n$, $\eta_{j}= \xi_{j}$, and $c_{j}= b_{j}$ for $1\leqslant j\leqslant n$ in equation (\ref{eqn: ext prod with comm}), we have the inequality
  \begin{equation*}
    \begin{aligned}
      \left\Vert \widetilde{\tau}(\sum_{j=1}^{n}\xi_{j}\otimes b_{j})\right\Vert^{2} &= \left\Vert \sum_{j=1}^{n}\widetilde{\phi}(\tau(\xi_{j})) b_{j}\right\Vert^{2} = \left\Vert \rho\bigg(\langle \sum_{j=1}^{n} \xi_{j}\otimes b_{j}, \sum_{j=1}^{n} \xi_{j}\otimes b_{j}\rangle_{A\otimes (B^{\sharp *}})'\bigg) \right\Vert \leqslant \left\Vert \sum_{j=1}^{n}\xi_{j} \otimes b_{j}\right\Vert^{2},
    \end{aligned}
  \end{equation*}
  so $\widetilde{\tau}$ extends to a continuous map on $\mathcal{H}\otimes (B^{\sharp *})'$, which we still denote by $\widetilde{\tau}$. Moreover, equation (\ref{eqn: ext prod with comm}) also shows that $(\widetilde{\phi}_{A}, \widetilde{\tau})$ satisfies relation (\ref{eqn: defn of rep of C* corresp, 2}) in the definition of a representation of C$^{*}$-correspondence. Since $\phi$ restricts to the identity map on $A$, using the relation (\ref{eqn: rel for toep, 1}) for $\tau$, we see that for $\xi \in \mathcal{H}$, $b\in (B^{\sharp *})'$, and simple tensors $a_{1}\otimes b_{1},a_{2}\otimes b_{2}\in A\odot (B^{\sharp *})'$,
  \begin{equation*}
    \begin{aligned}
      \widetilde{\tau}((a_{1}\otimes b_{1})(\xi\otimes b)(a_{2}\otimes b_{2})) &= \widetilde{\tau}((a_{1}\xi a_{2})\otimes (b_{1}bb_{2})) = \widetilde{\phi}(\tau(a_{1}\xi a_{2}))b_{1}bb_{2} = \widetilde{\phi}(a_{1}\tau(\xi )a_{2})b_{1}bb_{2}\\
      &= a_{1}b_{1}\widetilde{\phi}(\tau(\xi))b a_{2}b_{2} = \rho(a_{1}\otimes b_{1}) \widetilde{\tau}(\xi\otimes b)\rho(a_{2}\otimes b_{2}).
    \end{aligned}
  \end{equation*}

  Since the linear spans of simple tensors of the form $\xi\otimes b$ and $a\otimes b$ are norm dense in $\mathcal{H}\otimes (B^{\sharp *})'$ and $A\otimes (B^{\sharp *})'$, repsectively, and the maps $\rho$ and $\widetilde{\phi}$ are norm continuous, we conclude that the relation (\ref{eqn: defn of rep of C* corresp, 1}) also holds for $(\widetilde{\phi}_{A},\widetilde{\tau})$. Therefore, $(\widetilde{\phi}_{A},\widetilde{\tau})$ is indeed a representation of $\mathcal{H}\otimes (B^{\sharp *})'$ on $\mathbb{B}(\mathcal{K})$.

  By the universality of the Toeplitz-Pimsner algebras, the representation $(\rho,\widetilde{\tau})$ induces a $*$-homomorphism $\pi_{0}\colon \mathcal{T}(\mathcal{H}\otimes (B^{\sharp *})') \to \mathbb{B}(\mathcal{K})$. Since we can canonically identify $\mathcal{T}(\mathcal{H}\otimes (B^{\sharp *})')$ with $\mathcal{T}(\mathcal{H})\otimes (B^{\sharp *})'$, again by \cite{MR0512362}, $\pi_{0}$ restricts to a weakly nuclear $*$-homomorphism $\pi\colon \mathcal{T}(\mathcal{H}) \to (B^{\sharp *})'' = B^{\sharp *}$.
\end{proof}

\section{Weak exactness of amalgamated free product}

Let $M_{i}$, $i=1,2$, be $\sigma$-finite von Neumann algebras containing a copy of a unital von Neumann subalgebra $(B,\varphi)$ with a faithful normal state $\varphi$ on $B$. Assume that $M_{i}$ admits a faithful normal conditional expectation $E_{i}\colon M_{i}\to B$ for each $\in I$. Moreover, assume that each $M_{i}$ contains a unital ultraweakly dense C$^{*}$-algebra $A_{i}$ and a fixed unital C$^{*}$-subalgebra $D\subset A_{i}$ such that $E_{1}(A_{1}) = E_{2}(A_{2}) = D$. Let $M = M_{1}\overline{*}_{B} M_{2}$ be the amalgamated free product von Neumann algebra, $A = A_{1}*_{D}A_{2}$ be the (reduced) amalgamated free product C$^{*}$-algebra, $\varphi_{i} = \varphi\circ E_{i} \in (M_{i})_{*}$ for $i=1,2$, and $\widetilde{\varphi} = \varphi \circ E \in M_{*}$.

The following constructions are inspired by the arguments in \cite[Section 4.8]{MR2391387}. The normal u.c.p.\ map
\[\phi\colon M_{1}\oplus M_{2} \ni (m_{1},m_{2})\mapsto (E_{2}(m_{2}),E_{1}(m_{1})) \in B\oplus B \subset M_{1}\oplus M_{2}\]
restricts to a u.c.p.\ map $A_{1}\oplus A_{2} \ni (a_{1} \oplus a_{2})\mapsto (E_{2}(a_{2}) \oplus E_{1}(a_{1})) \in D\oplus D \subset A_{1}\oplus A_{2}$. We let $\mathcal{H} = \mathcal{H}^{\phi}_{A_{1}\oplus A_{2}}$ denote the Hilbert C$^{*}$-correspondence over $A_{1}\oplus A_{2}$ as given in Example \ref{exmp: correspondence}, so $\mathcal{H}$ is $(A_{1}\oplus A_{2}\subset M_{1}\oplus M_{2})$-normal.

For simplicity, we let $L^{2}B = L^{2}(B,\varphi)$, $L^{2}M_{i} = L^{2}(M_{i},\varphi_{i})$, $L^{2}M_{i}^{o} = L^{2}M_{i} \ominus L^{2}B$, $\xi_{i} \in L^{2}M_{i}$ be the implementing vector of $\varphi_{i}$, and $A_{i}^{o} = A_{i}\cap \ker(E_{i})$. We consider the Hilbert space $\mathcal{K} = \mathcal{K}_{1}\oplus \mathcal{K}_{2}$, where
\[\mathcal{K}_{i} = \bigoplus_{n\geqslant 1}\bigoplus_{i_{1}=i,i_{1}\neq\cdots\neq i_{n}}L^{2}M_{i_{1}}\otimes_{B}\cdots \otimes_{B} L^{2}M_{i_{n}}.\]
Let $T\in\mathbb{B}(\mathcal{K})$ be the isometry that exchanges $\mathcal{K}_{1}$ and $\mathcal{K}_{2}$ and acts like a creation operator:
\[T(\zeta_{1}\otimes \cdots \otimes \zeta_{n}) = \xi_{i}\otimes \zeta_{1}\otimes\cdots \otimes \zeta_{n},\]
where $i=1$ or $2$ is suitably chosen so that if $\zeta_{1}\in \mathcal{K}_{i_{1}}$ then $i \neq i_{1}$. The adjoint $T^{*}$ satisfies the property that $T^{*}(\zeta) = 0$ for $\zeta \in L^{2}M_{1}\cup L^{2}M_{2}$ and 
\[T^{*}(m\otimes \zeta_{2}\otimes\cdots \otimes\zeta_{n}) = E_{i}(m)\zeta_{2}\otimes\cdots \otimes\zeta_{n}\]
for $m\in M_{i}$ and $\zeta_{2}\otimes \cdots\otimes \zeta_{n} \in L^{2}M_{i_{2}} \otimes_{B}\cdots \otimes_{B} L^{2}M_{i_{n}}$ with $i\neq i_{2}\neq\cdots \neq i_{n}$.

Regarding each $M_{i}$ as a unital von Neumann subalgebra of $\mathbb{B}(\mathcal{K}_{i})$ acting on the first tensor component gives an ultraweakly continuous unital faithful $*$-representation $\iota\colon A_{1}\oplus A_{2}\to \mathbb{B}(\mathcal{K})$ that extends to a normal faithful representation on $M_{1}\oplus M_{2}$. Let $\tau\colon \mathcal{H} \to \mathbb{B}(\mathcal{K})$ be a linear map defined by 
\[\tau((a_{1}\oplus a_{2})\otimes (b_{1} \oplus b_{2})) = (a_{1}\oplus a_{2})T(b_{1}\oplus b_{2}), \quad (a_{1}\oplus a_{2})\otimes (b_{1} \oplus b_{2})\in \mathcal{H}.\]

\begin{lem}
  $(\iota, \tau)$ is a representation of the C$^{*}$-correspondence $\mathcal{H}$ on $\mathbb{B}(\mathcal{K})$.
\end{lem}
\begin{proof}
  Notice that we have
  \[T^{*}(a_{1}\oplus a_{2})T = E_{2}(a_{2})\oplus E_{1}(a_{1}), \qquad (a_{1}\oplus a_{2}) \in A_{1}\oplus A_{2},\]
  so for $a_{1,i},b_{1,i},c_{1,j},d_{1,j}\in A_{1}$ and $a_{2,i},b_{2,i},c_{2,j},d_{2,j}\in A_{2}$ for $1\leqslant i \leqslant n$ and $1\leqslant j \leqslant m$,
  \begin{equation} \label{eqn: rep tau well-defn}
    \begin{aligned}
      &\tau\bigg(\sum_{i=1}^{n}(a_{1,i}\oplus a_{2,i})\otimes(b_{1,i}\oplus b_{2,i})\bigg)^{*}\tau\bigg(\sum_{j=1}^{m}(c_{1,j}\oplus c_{2,j})\otimes(d_{1,j}\oplus d_{2,j})\bigg)\\
      =& \sum_{i,j} (b_{1,i}^{*}\oplus b_{2,i}^{*})T^{*}(a_{1,i}^{*}\oplus a_{2,i}^{*})(c_{1,j}\oplus c_{2,j})T(d_{1,j}\oplus d_{2,j})\\
      =& \sum_{i,j} (b_{1,i}^{*}\oplus b_{2,i}^{*}) (E_{2}(a_{2,i}^{*}c_{2,j})\oplus E_{1}(a_{1,i}^{*}c_{1,j}))(d_{1,j}\oplus d_{2,j})\\
      =& \sum_{i,j} (b_{1,i}^{*}E_{2}(a_{2,i}c_{2,j})d_{1,j})\oplus  (b_{2,i}^{*}E_{2}(a_{1,i}c_{1,j})d_{2,j})\\
      =& \sum_{i,j}\iota(\langle (a_{1,i}\oplus a_{2,i})\otimes(b_{1,i}\oplus b_{2,i}), (c_{1,j}\oplus c_{2,j})\otimes(d_{1,j}\oplus d_{2,j})\rangle_{A_{1}\oplus A_{2}})\\
      =& \iota\bigg(\left\langle \sum_{i=1}^{n}(a_{1,i}\oplus a_{2,i})\otimes(b_{1,i}\oplus b_{2,i}), \sum_{j=1}^{m}(c_{1,j}\oplus c_{2,j})\otimes(d_{1,j}\oplus d_{2,j})\right\rangle_{A_{1}\oplus A_{2}}\bigg).
    \end{aligned}
  \end{equation}
  In particular, taking $n=m$, $a_{i,j}=c_{i,j}$, and $b_{i,j}=d_{i,j}$ for $i=1,2$ and $1\leqslant j\leqslant n$, we have
  \[\left\Vert \tau\bigg(\sum_{i=1}^{n}(a_{1,i}\oplus a_{2,i})\otimes(b_{1,i}\oplus b_{2,i})\bigg)\right\Vert^{2} = \left\Vert \sum_{i=1}^{n}(a_{1,i}\oplus a_{2,i})\otimes(b_{1,i}\oplus b_{2,i})\right\Vert^{2},\]
  so $\tau$ is well-defined and continuous. Relation (\ref{eqn: defn of rep of C* corresp, 1}) clearly holds by definition of $\tau$, and (\ref{eqn: defn of rep of C* corresp, 2}) follows from (\ref{eqn: rep tau well-defn}) above. Therefore, $(\iota, \tau)$ is a representation of $\mathcal{H}$ on $\mathbb{B}(\mathcal{K})$.
\end{proof}

One can easily check that the gauge-invariant uniqueness theorem applies to the representation $(\iota,\tau)$, so the C$^{*}$-algebra $\mathcal{T}(\mathcal{H})$ is canonically $*$-isomorphic to the C$^{*}$-subalgebra $C^{*}(A_{1}\oplus A_{2},T)$ of $\mathbb{B}(\mathcal{K})$.

We define $p \coloneqq 1- T^{2}(T^{*})^{2}\in C^{*}(A_{1}\oplus A_{2},T)$, which is the orthogonal projection onto
\[\bigoplus_{i\neq j} \bigg(L^{2}M_{i} \oplus (L^{2}M_{i}\otimes_{B} L^{2}M_{j}) \oplus \big(\big((L^{2}M_{i}\otimes_{B} L^{2}M_{j}) \ominus (B\xi_{i} \otimes_{B} B\xi_{j})\big)\otimes_{B} \mathcal{K}_{i}\big)\bigg).\]

For $i,j=1,2$, we define following subspaces of $p\mathcal{K}$:
\begin{equation*}
  \begin{aligned}
    \mathcal{K}_{1,1} &= L^{2}M_{1} \oplus \bigoplus L^{2}M_{1}\otimes_{B} L^{2}M_{2}^{o} \otimes_{B} \cdots \otimes_{B} L^{2}M_{2}^{o} \otimes_{B} L^{2}M_{1}\\
    \mathcal{K}_{1,2} &= (L^{2}M_{1}\otimes_{B} L^{2}M_{2}) \oplus \bigoplus L^{2}M_{1}\otimes_{B} L^{2}M_{2}^{o} \otimes_{B} \cdots \otimes_{B} L^{2}M_{1}^{o} \otimes_{B} L^{2}M_{2}\\
    \mathcal{K}_{2,1} &= (L^{2}M_{2}\otimes_{B} L^{2}M_{1}) \oplus \bigoplus L^{2}M_{2}\otimes_{B} L^{2}M_{1}^{o} \otimes_{B} \cdots \otimes_{B} L^{2}M_{2}^{o} \otimes_{B} L^{2}M_{1}\\
    \mathcal{K}_{2,2} &= L^{2}M_{2} \oplus \bigoplus L^{2}M_{2}\otimes_{B} L^{2}M_{1}^{o} \otimes_{B} \cdots \otimes_{B} L^{2}M_{1}^{o} \otimes_{B} L^{2}M_{2}
  \end{aligned}
\end{equation*}
Namely, $\mathcal{K}_{i,j}$ is the subspace consisting of the tensors beginning with vectors from $L^{2}M_{i}$, ending with those from $L^{2}M_{j}$, and in the middle vectors are from the orthogonal complement of $B\xi_{1}, B\xi_{2}$, if there are any. Via the standard identifications $\xi_{1}B \otimes_{B} L^{2}M_{2}^{o} \simeq L^{2}M_{2}^{o}$ and $L^{2}M_{2}^{o} \otimes_{B} \xi_{1}B \simeq L^{2}M_{2}^{o}$, we have a natural isomorphism $V_{1,1}\colon \mathcal{K}_{1,1}\to L^{2}M$. Similarly, by the identification $\xi_{2}B \otimes_{B} L^{2}M_{1}^{o} \simeq L^{2}M_{1}^{o}$ and $L^{2}M_{1}^{o} \otimes_{B} \xi_{2}B \simeq L^{2}M_{2}^{o}$, we have natural identification maps $V_{i,j}\colon \mathcal{K}_{i,j}\to L^{2}M = L^{2}(M,\widetilde{\varphi})$ for $i,j=1,2$.

Let $u = p(T+ T^{*})p \in C^{*}(A,T)$. As seen in \cite[Section 4.8]{MR2391387}, $u$ is a self-adjoint partial isometry with $u^{2} = p$ that interchanges $p\mathcal{K}_{1}$ and $p\mathcal{K}_{2}$, via the canonical identifications among subspaces of $p\mathcal{K}$ as indicated in the following diagram:
\begin{equation*} 
  \begin{tikzcd}
    L^{2}M_{1} \arrow[dr, leftrightarrow] & \xi_{1}B\otimes_{B}L^{2}M_{2}\arrow[dl, crossing over,leftrightarrow]&  & L^{2}M_{1}^{o}\otimes_{B}L^{2}M_{2} \arrow[dr,leftrightarrow] & \xi_{1}B\otimes_{B}L^{2}M_{2}^{o}\otimes_{B} L^{2}M_{1}\arrow[dl, crossing over,leftrightarrow] & \cdots\\
    L^{2}M_{2}  & \xi_{2}B\otimes_{B}L^{2}M_{1}&  & L^{2}M_{2}^{o}\otimes_{B}L^{2}M_{1}  & \xi_{2}B\otimes_{B}L^{2}M_{1}^{o}\otimes_{B} L^{2}M_{2} & \cdots
\end{tikzcd}
\end{equation*}
The two-sided arrows in the above diagram represent canonical identification maps. In particular, from above diagram we see that $u$ interchanges $\mathcal{K}_{1,1}$ with $\mathcal{K}_{2,1}$, and $\mathcal{K}_{1,2}$ with $\mathcal{K}_{2,2}$ via canonical identifications, i.e., for any $\xi \in L^{2}M$, 
\begin{equation} \label{graph: u action}
  u(V_{1,1}^{*}\xi) = V_{2,1}^{*}\xi\,, \quad u(V_{2,1}^{*}\xi) = V_{1,2}^{*}\xi\,,\quad u(V_{1,2}^{*}\xi) = V_{2,2}^{*}\xi\,,\quad u(V_{2,2}^{*}\xi) = V_{1,2}^{*}\xi.
\end{equation}

We define a u.c.p.\ map $\Phi\colon \mathbb{B}(\mathcal{K})\to \mathbb{B}(L^{2}M)$ by $\Phi(x) = \frac{1}{4}\sum_{i,j,k,l=1}^{2} V_{i,j}xV_{k,l}^{*}$, so that the support projection of $\Phi$ is no larger than the orthgonal projection from $\mathcal{K}$ onto $\bigoplus_{i,j=1,2}\mathcal{K}_{i,j}$. 
Therefore, $p\geqslant \supp(\Phi)$ and thus $\Phi(p) = 1$. We also have 
\begin{equation*}
	\Phi((m_{1}\oplus m_{2})) = \frac{1}{2}\lambda_{1}(m_{1})+\frac{1}{2}\lambda_{2}(m_{2}) \in M\,, \qquad (m_{1}\oplus m_{2})\in M_{1}\oplus M_{2},
\end{equation*}
and in particular when we restrict to $A_{1}\oplus A_{2}$,
\begin{equation} \label{eqn: Phi(a_1 oplus a_2)}
	\Phi((a_{1}\oplus a_{2})) = \frac{1}{2}\lambda_{1}(a_{1})+\frac{1}{2}\lambda_{2}(a_{2}) \in A\,, \qquad (a_{1}\oplus a_{2})\in A_{1}\oplus A_{2}.
\end{equation}
Therefore, $\Phi$ is a normal u.c.p.\ map that maps $A_{1}\oplus A_{2}$ into $A$, and hence by Lemma \ref{lem: ext cont weak top}, $\Phi$ extends to a normal u.c.p.\ map 
\[\Phi^{\sharp *}\colon (\mathbb{B}(\mathcal{K})^{(A_{1}\oplus A_{2})\sharp (A_{1}\oplus A_{2})})^{*} \to (\mathbb{B}(L^{2}M)^{A\sharp A})^{*}\]
satisfying
\[i_{\mathbb{B}(L^{2}M)} \circ \Phi = \Phi^{\sharp *} \circ i_{\mathbb{B}(\mathcal{K})},\]
where $i_{\mathbb{B}(\mathcal{K})}\colon \mathbb{B}(\mathcal{K})\to  (\mathbb{B}(\mathcal{K})^{(A_{1}\oplus A_{2})\sharp (A_{1}\oplus A_{2})})^{*}$ and $i_{\mathbb{B}(L^{2}M)}\colon \mathbb{B}(L^{2}M)\to  (\mathbb{B}(L^{2}M)^{A\sharp A})^{*}$ are the canonical inclusion maps. Hence, if we let $p_{\nor} \in \mathbb{B}(\mathcal{K})^{**}$ denote the projection corresponding to the identity representation $A_{1}\oplus A_{2}\subset M_{1}\oplus M_{2}$ and $q_{\nor} \in \mathbb{B}(L^{2}M)^{**}$ denote the projection corresponding to the identity representation $A\subset M$, then $\Phi^{\sharp *}(p_{\nor}) = \Phi^{\sharp *} \circ i_{\mathbb{B}(\mathcal{K})}(1) = i_{\mathbb{B}(L^{2}M)} \circ \Phi(1) = q_{\nor}$. 

Since the partial isometry $u$ interchanges $\mathcal{K}_{1,1}$ with $\mathcal{K}_{2,1}$, and $\mathcal{K}_{1,2}$ with $\mathcal{K}_{2,2}$ via canonical identifications, by (\ref{graph: u action}) we see that for any $\xi \in L^{2}M$, 
\[\Phi(u)(\xi) = \frac{1}{4}\sum_{i,j,=1}^{2} V_{i,j}\bigg(\sum_{k,l=1}^{2}uV_{k,l}^{*}\xi\bigg) = \frac{1}{4}\sum_{i,j=1}^{2} V_{i,j}\sum_{k,l=1}^{2}V_{k,l}^{*}\xi = \xi.\]

Since $\Phi$ maps $p$ to the identity $1\in \mathbb{B}(L^{2}M)$, we have 
\begin{equation} \label{eqn: Phi(u)=1}
\Phi(T+T^{*}) = \Phi(p(T+T^{*})p) = \Phi(u) = 1.
\end{equation} 
It also follows from (\ref{graph: u action}) that $\Phi(u(m_{1}\oplus m_{2})u) = \frac{1}{2}\lambda_{1}(m_{1})+\frac{1}{2}\lambda_{2}(m_{2})$ for $(m_{1}\oplus m_{2}) \in M_{1}\oplus M_{2}$.

For $i=1,2$, we define u.c.p.\ maps $\psi_{i}\colon A_{i}\to pC^{*}(A_{1}\oplus A_{2},T)p$ by
\[\psi_{i}(a) = pap+ uau, \quad a\in A_{i}.\]
It is straightforward to see that $\psi_{1}(b) = \psi_{2}(b)$ for $b\in B$. By \cite[Theorem 4.8.2]{MR2391387}, there exists a u.c.p.\ map $\Psi\colon A \to pC^{*}(A_{1}\oplus A_{2},T)p\subset \mathcal{T}(\mathcal{H})$ such that $\Psi(b) = \psi_{1}(b) = \psi_{2}(b)$ for $b\in B$ and
\[\Psi(a_{1}\cdots a_{n}) = \psi_{i_{1}}(a_{1})\cdots \psi_{i_{n}}(a_{n})= \Psi(a_{1})\cdots \Psi(a_{n})\]
for $a_{k}\in A^{o}_{i_{k}}$ with $i_{1}\neq \cdots \neq i_{k}$. Moreover, for every $a\in A_{i}$, 
\[\Phi\circ \Psi(a) = \Phi(pap+uau) = \Phi(pap) + \Phi(uau) = \frac{1}{2}\lambda_{i}(a) + \frac{1}{2}\lambda_{i}(a) = \lambda_{i}(a).\]
Therefore, the u.c.p.\ map $\Phi\circ \Psi$ is the canonical inclusion map when restricted to $A_{1}\cup A_{2}$. Since $A_{1}$ and $A_{2}$ generate $A = A_{1}*_{D}A_{2}$ as C$^{*}$-algebra, $\Phi\circ \Psi\colon A\to \mathbb{B}(L^{2}M)$ is in fact the canonical inclusion map.

Since $A_{1}\oplus A_{2}$ is weakly exact in $M_{1}\oplus M_{2}$, by \cite[Theorem 5.1]{ding2023biexact}, the canonical inclusion $A_{1}\oplus A_{2}\subset \mathbb{B}(\mathcal{K})$ is $(A_{1}\oplus A_{2}\subset M_{1}\oplus M_{2})$-nuclear. Applying Lemma \ref{lem: normal hom ext hom on toep} to the canonical inclusion $\mathcal{T}(\mathcal{H}) \subset \mathbb{B}(\mathcal{K})$ we obtain a weakly nuclear $*$-homomorphism $\pi\colon \mathcal{T}(\mathcal{H})\to (\mathbb{B}(\mathcal{K})^{(A_{1}\oplus A_{2})\sharp (A_{1}\oplus A_{2})})^{*}$ such that $\pi(x) = i_{\mathbb{B}(\mathcal{K})}(x)$ for all $x\in A\cup \tau(\mathcal{H})\cup \tau(\mathcal{H})^{*}$.

Combining all the maps derived above, we obtain the following diagram, which is commutative only in solid arrows:

\begin{equation*}
\begin{tikzcd}[column sep=small]
& &(\mathbb{B}(\mathcal{K})^{(A_{1}\oplus A_{2})\sharp (A_{1}\oplus A_{2})})^{*} \arrow[rr, "\Phi^{\sharp *}"]   & & (\mathbb{B}(L^{2}M)^{A\sharp A})^{*} \\
A \arrow[r, "\Psi"] \arrow[rrrr, bend right=17, "\subset", "\text{canonical inclusion}" swap] &\mathcal{T}(\mathcal{H}) \arrow[rr, hook, "\subset", "\text{canonical inclusion}" swap] \arrow[ur, dashed, "\pi"]& & \mathbb{B}(\mathcal{K}) \arrow[r, "\Phi"]  \arrow[ul, "i_{\mathbb{B}(\mathcal{K})}", labels=above right] & \mathbb{B}(L^{2}M) \arrow[u, "i_{\mathbb{B}(L^{2}M)}", labels=right]
\end{tikzcd}
\end{equation*}

Although the above diagram may no longer be commutative if we include the dahsed arrow $\pi$, we will show in (the proof of) the next theorem that we do have
\[\Phi^{\sharp *}\circ \pi \circ \Psi = i_{\mathbb{B}(L^{2}M)} \circ \Phi \circ \Psi.\]

\begin{thm} \label{thm: amalg free prod wkly exact subalg}
  For $i=1,2$, let $(M_{i}, \varphi_{i})$ be a von Neumann algebra with a faithful normal state $\varphi_{i}$, each containing a copy of a fixed von Neumann subalgebra $B$ and admitting a conditional expectation $E_{i}\colon M_{i}\to B$. Suppose each $M_{i}$ contains an ultraweakly dense unital C$^{*}$-algebra $A_{i}$, and each $A_{i}$ contains a copy of a fixed C$^{*}$-subalgebra $D$ such that $E_{i}(A_{i}) = D$ for $i=1,2$. 

  If $A_{i}$ is weakly exact in $M_{i}$ for $i=1,2$, then $A= A_{1}*_{D} A_{2}$ is weakly exact in $M=M_{1}*_{B} M_{2}$.
\end{thm}

\begin{proof}
  Keeping the same notations as in the constructions preceding the theorem, we claim that it suffices to show 
  \begin{equation} \label{eqn: thm comm diagram}
  	\Phi^{\sharp *}\circ \pi \circ \Psi = i_{\mathbb{B}(L^{2}M)} \circ \Phi \circ \Psi 
  \end{equation}
  on $A$. Indeed, if this is true, then $i_{\mathbb{B}(L^{2}M)} \circ \Phi \circ \Psi $ is a weakly nuclear u.c.p from $A$ into $(\mathbb{B}(L^{2}M)^{A\sharp A})^{*}$ because $\pi$ is weakly nuclear and $\Phi^{\sharp *}$ is normal. Then by \cite[Lemma 4.4]{ding2023biexact}, the u.c.p.\ map $\Phi \circ \Psi\colon A\to  \mathbb{B}(L^{2}M)$, which is just the canonical inclusion map, is $(A \subset M)$-nuclear. By \cite[Theorem 5.1]{ding2023biexact}, it follows that $A$ is weakly exact in $M$.

  To show equation (\ref{eqn: thm comm diagram}), we first recall that $\Phi(1- T^{2}(T^{*})^{2}) = \Phi(p) = 1$, so $\Phi(T^{2}(T^{*})^{2}) = 0$. Applying Lemma \ref{lem: p_nor invariant} to $\mathcal{T}(\mathcal{H}) = C^{*}(A_{1}\oplus A_{2},T)$, we see that
  \begin{equation*}
  	\begin{aligned}
  		0\leqslant \Phi^{\sharp *}(\pi(T^{2}(T^{*})^{2})) &= \Phi^{\sharp *}(\pi(T)^{2}\pi(T^{*})^{2})\\
  		&= \Phi^{\sharp *}((p_{\nor}T p_{\nor})^{2}(p_{\nor}T^{*}p_{\nor})^{2})\\
  		&= \Phi^{\sharp *}(p_{\nor}T^{2}p_{\nor} (T^{*})^{2}p_{\nor})\\
  		&\leqslant \Phi^{\sharp *}(p_{\nor}T^{2} (T^{*})^{2}p_{\nor})\\
  		&= \Phi^{\sharp *}(i_{\mathbb{B}(\mathcal{K})}(T^{2}(T^{*})^{2}))\\
  		&= i_{\mathbb{B}(L^{2}M)}(\Phi(T^{2}(T^{*})^{2}))\\
  		&= 0,
  	\end{aligned}
  \end{equation*}
  so $\Phi^{\sharp *}(\pi(T^{2}(T^{*})^{2})) = 0$. Thus we have
  \begin{equation} \label{eqn: Phi(phi(p)) = q_nor}
    \Phi^{\sharp *}(\pi(p)) = \Phi^{\sharp *}(\pi (1-T^2 (T^{*})^{2})) = \Phi^{\sharp *}(\pi(1)) -  \Phi^{\sharp *}(\pi(T^{2}(T^{*})^{2})) = \Phi^{\sharp *}(\pi(1)) = q_{\nor}.
  \end{equation}
	In particular, $\pi(p)$ is in the multiplicative domain of $\Phi^{\sharp *}$. Also, by (\ref{eqn: Phi(u)=1})
  \begin{equation} \label{eqn: Phi(pi(u)) = q_nor}
  	\begin{aligned}
    	\Phi^{\sharp *}(\pi (u)) &= \Phi^{\sharp *}(\pi(p(T+T^{*})p)) = \Phi^{\sharp *}(\pi(p)\pi(T+T^{*})\pi(p)) = \Phi^{\sharp *}(\pi(T+T^{*})) \\
    	&= \Phi^{\sharp *}(i_{\mathbb{B}(\mathcal{K})}(T+T^{*})) = i_{\mathbb{B}(L^{2}M)}\circ \Phi(T+ T^{*}) = q_{\nor},
  	\end{aligned}
  \end{equation}
  so $\pi(u)$ also lies in the multiplicative domain of $\Phi^{\sharp *}$.

  Since $\Phi\circ \Psi\colon A\to \mathbb{B}(L^{2}M)$ is the canonical inclusion map, $i_{\mathbb{B}(L^{2}M)} \circ \Phi \circ \Psi \colon A\to (\mathbb{B}(L^{2}M)^{A\sharp A})^{*}$ is a $*$-homomorphism and for $i=1,2$ and any $a\in A_{i}$, we have
  \[(i_{\mathbb{B}(L^{2}M)} \circ \Phi \circ \Psi)(a) = i_{\mathbb{B}(L^{2}M)} \circ\lambda_{i}(a).\]
  On the other hand, for $a\in A_{i}$, by (\ref{eqn: Phi(phi(p)) = q_nor}), (\ref{eqn: Phi(pi(u)) = q_nor}), and (\ref{eqn: Phi(a_1 oplus a_2)}),
  \begin{equation*} 
  	\begin{aligned}
    	(\Phi^{\sharp *}\circ \pi\circ \Psi)(a) &= \Phi^{\sharp *}(\pi(pap+uau)) 
    	= 2 \Phi^{\sharp *}(\pi(a))
    	=  2 \Phi^{\sharp *}(i_{\mathbb{B}(\mathcal{K})}(a))
    	= 2 i_{\mathbb{B}(L^{2}M)}(\Phi(a))\\
    	&= 2 i_{\mathbb{B}(L^{2}M)}(\frac{1}{2}\lambda_{i}(a))
    	= i_{\mathbb{B}(L^{2}M)}(\lambda_{i}(a))
    	= (i_{\mathbb{B}(L^{2}M)} \circ \Phi \circ \Psi)(a).
    \end{aligned}
  \end{equation*}
  In particular, $a$ is in the multiplicative domain of $\Phi^{\sharp *}\circ \pi\circ \Psi$. Since $A_{1}$ and $A_{2}$ generate $A = A_{1}*_{D}A_{2}$ as C$^{*}$-algebra, we conclude that equation (\ref{eqn: thm comm diagram}) holds on $A$.
\end{proof}

\begin{proof}[Proof of Theorem \ref{thm: amalg free prod preserves wk exact}]
  When $I = \{1,2\}$, the corollary is a direct consequence of Theorem \ref{thm: amalg free prod wkly exact subalg} and Theorem \cite[Theorem 1.0.1]{MR3004955} applied to $A_{1} = M_{1}, A_{2} = M_{2}$, and $D= B$. The case when $I$ is a finite set follows from induction and the fact that taking amalgamated free product is associative. 

  For the general case, it suffices to notice that by \cite[Proposition 4.12]{MR3004955} and \cite[Corollary 4.9, 5.6]{ding2023biexact}, increasing unions of weakly exact von Neumann algebras are again weakly exact.
\end{proof}

\begin{rem}
	The constructions in Section 3 can be easily adapted to recover the result in \cite{MR2039095} that an amalagamated free product of exact C$^{*}$ algebras is exact, which is similar to the proof given in \cite[Section 4.8]{MR2391387}. Indeed, if $A_{1}$ and $A_{2}$ are exact C$^{*}$-algebras containing a fixed unital C$^{*}$-subalgebra $D$ and admit faithful (more generally, nondegenerate) conditional expectations $E_{i}\colon A_{i}\to D$ for $i=1,2$, then we let $\mathcal{K}$ be the C$^{*}$-correspondence over $A_{1}\oplus A_{2}$ given by
\[\mathcal{K} = \bigoplus_{n\geqslant 1}\bigoplus_{i_{1}\neq \cdots \neq i_{n}} L^{2}A_{i_{1}} \otimes_{D} \cdots \otimes_{D} L^{2}A_{i_{1}}\]
and $\mathcal{F} = *_{D}(L^{2}A_{i})_{i=1,2}$ be the free product Hilbert $D$-module (see Section 2.1). We regard $\mathbb{B}(\mathcal{K})$ and $\mathbb{B}(\mathcal{F})$ as $\mathbb{C}$-C$^{*}$-algebras and consider the $\mathbb{C}$-topology, namely the norm topology, instead of $(A\subset M)$-topologies. Then by the same calculations we derive the following commutative diagram

\begin{equation*}
\begin{tikzcd}[column sep=small]
& &\mathbb{B}(\mathcal{K})^{**} \arrow[rr, "\Phi^{\sharp *}"]   & & \mathbb{B}(\mathcal{F})^{**} \\
A=A_{1}*_{D}A_{2} \arrow[r, "\Psi"] \arrow[rrrr, bend right=17, "\subset", "\text{canonical inclusion}" swap] &\mathcal{T}(\mathcal{H}) \arrow[rr, hook, "\subset", "\text{canonical inclusion}" swap] \arrow[ur,  "\pi"]& & \mathbb{B}(\mathcal{K}) \arrow[r, "\Phi"]  \arrow[ul, "i_{\mathbb{B}(\mathcal{K})}", labels=above right] & \mathbb{B}(\mathcal{F}) \arrow[u, "i_{\mathbb{B}(\mathcal{F})}", labels=right]
\end{tikzcd}
\end{equation*}

with $\pi$ still a weakly nuclear $*$-homomorphism. Therefore, the inclusion $A_{1}*_{D}A_{2}\subset \mathbb{B}(\mathcal{F})$ is nuclear, so the amalgamated free product $A_{1}*_{D}A_{2}$ is exact.

\end{rem}

We recall the HNN extensions of von Neumann algebras in \cite{MR2152505} and \cite{MR2546003}. Let $M$ be a $\sigma$-finite von Neumann algebra and $B\subset M$ be a $\sigma$-finite von Neumann subalgebra with a faithful normal conditional expectation $E_{B}^{M}\colon M\to B$. Let $\theta\colon B\to M$ be a faithful normal unital $*$-homomorphism and assume there exists a faithful normal conditional expectation $E_{\theta(B)}^{M}\colon N\to \theta(B)$. The \textit{HNN extension of $M$ by $\theta$ with respect to $E^{M}_{B}, E^{M}_{\theta(B)}$} is the unique triple $(N, E^{N}_{M}, u)$ consisting of a von Neumann algebra $N$ containing $M$, a faithful normal conditional expectation $E^{N}_{M}\colon N\to M$, and a unitary $u = u(\theta)$ in $N$ such that $u\, \theta(b) u^{*} = b$ for every $b\in B$ and $E^{N}_{M}(w) = 0$ for every reduced word $w$ in $ M$ and $u$. Here if we let $B_{1} = B$ and $B_{-1} = \theta(B)$, then a word $w = x_{0}u^{\varepsilon_{1}}x_{1}\cdots x_{n-1}u^{\varepsilon_{n}}x_{n}$ with $x_{i}\in M$ and $\varepsilon_{i}\in \{-1,1\}$ is called \textit{reduced} if $x_{i} \in M\ominus B_{\varepsilon_{i}}$ whenever $\varepsilon_{i}\neq \varepsilon_{i+1}$. 

\begin{proof}[Proof of Corollary \ref{cor: wk exact for HNN}]
  We consider $B\oplus B$ as a von Neumann subalgebra of $M\otimes \mathbb{M}_{2}(\mathbb{C})$ via the inclusion 
\[(b_{1}\oplus b_{2}) \mapsto \begin{bmatrix} b_{1} & 0 \\ 0 & \theta(b_{2})\end{bmatrix},\]
so then there exists a faithful normal conditional expectation $E_{\theta}\colon M\otimes \mathbb{M}_{2}(\mathbb{C}) \to B\oplus B$ given by
\[E_{\theta} = \begin{bmatrix} E^{M}_{B} & 0 \\ 0 & E^{M}_{\theta(B)}\end{bmatrix}.\]

Similarly, we also embed $B\oplus B$ into $B\otimes \mathbb{M}_{2}(\mathbb{C})$ via 
\[(b_{1}\oplus b_{2}) \mapsto \begin{bmatrix} b_{1} & 0 \\ 0 & b_{2}\end{bmatrix}\]
with a faithful normal conditional expectation $E_{1}\colon B\otimes \mathbb{M}_{2}(\mathbb{C}) \to B\oplus B$ defined by
\[E_{1} = \begin{bmatrix} \id_{B} & 0 \\ 0 & \id_{B}\end{bmatrix}.\]

Consider the amalgamated free product $(\widetilde{N}, E) = (M\otimes \mathbb{M}_{2}(\mathbb{C}),E_{\theta})\overline{*}_{B\oplus B}(B\otimes \mathbb{M}_{2}(\mathbb{C}), E_{1})$. Since $M$ and $B$ are weakly exact, by Theorem \ref{thm: amalg free prod preserves wk exact}, $\widetilde{N}$ is also weakly exact.

Let $N$ denote the HNN extension of $M$ by $\theta$. By \cite[Proposition 3.1]{MR2546003}, there exists a projection $p\in \widetilde{N}$ such that $p\widetilde{N}p \simeq N$. Since $\widetilde{N}$ is weakly exact, we conclude that $N$ is weakly exact as well.
\end{proof}

\bibliographystyle{amsalpha}
\bibliography{ref}

\providecommand{\bysame}{\leavevmode\hbox to3em{\hrulefill}\thinspace}
\providecommand{\MR}{\relax\ifhmode\unskip\space\fi MR }
\providecommand{\MRhref}[2]{%
  \href{http://www.ams.org/mathscinet-getitem?mr=#1}{#2}
}
\providecommand{\href}[2]{#2}
\begin{thebibliography}{CdSS16}

\bibitem[BIP21]{MR4258166}
R\'{e}mi Boutonnet, Adrian Ioana, and Jesse Peterson, \emph{Properly proximal
  groups and their von {N}eumann algebras}, Ann. Sci. \'{E}c. Norm. Sup\'{e}r.
  (4) \textbf{54} (2021), no.~2, 445--482. \MR{4258166}

\bibitem[BO08]{MR2391387}
Nathanial~P. Brown and Narutaka Ozawa, \emph{{$C^*$}-algebras and
  finite-dimensional approximations}, Graduate Studies in Mathematics, vol.~88,
  American Mathematical Society, Providence, RI, 2008. \MR{2391387}

\bibitem[CdSS16]{MR3494487}
Ionut Chifan, Rolando de~Santiago, and Thomas Sinclair, \emph{{$W^*$}-rigidity
  for the von {N}eumann algebras of products of hyperbolic groups}, Geom.
  Funct. Anal. \textbf{26} (2016), no.~1, 136--159. \MR{3494487}

\bibitem[CS13]{MR3087388}
Ionut Chifan and Thomas Sinclair, \emph{On the structural theory of {${\rm
  II}_1$} factors of negatively curved groups}, Ann. Sci. \'{E}c. Norm.
  Sup\'{e}r. (4) \textbf{46} (2013), no.~1, 1--33. \MR{3087388}

\bibitem[DEP23]{MR4675043}
Changying Ding, Srivatsav~Kunnawalkam Elayavalli, and Jesse Peterson,
  \emph{Properly proximal von {N}eumann algebras}, Duke Math. J. \textbf{172}
  (2023), no.~15, 2821--2894. \MR{4675043}

\bibitem[DP23]{ding2023biexact}
Changying Ding and Jesse Peterson, \emph{Biexact von neumann algebras},
  arXiv:2309.10161 (2023).

\bibitem[DS01]{MR1880402}
Kenneth~J. Dykema and Dimitri Shlyakhtenko, \emph{Exactness of
  {C}untz-{P}imsner {$C^*$}-algebras}, Proc. Edinb. Math. Soc. (2) \textbf{44}
  (2001), no.~2, 425--444. \MR{1880402}

\bibitem[Dyk04]{MR2039095}
Kenneth~J. Dykema, \emph{Exactness of reduced amalgamated free product
  {$C^*$}-algebras}, Forum Math. \textbf{16} (2004), no.~2, 161--180.
  \MR{2039095}

\bibitem[FMR03]{MR1986889}
Neal~J. Fowler, Paul~S. Muhly, and Iain Raeburn, \emph{Representations of
  {C}untz-{P}imsner algebras}, Indiana Univ. Math. J. \textbf{52} (2003),
  no.~3, 569--605. \MR{1986889}

\bibitem[HI16]{MR3555359}
Cyril Houdayer and Yusuke Isono, \emph{Bi-exact groups, strongly ergodic
  actions and group measure space type {III} factors with no central sequence},
  Comm. Math. Phys. \textbf{348} (2016), no.~3, 991--1015. \MR{3555359}

\bibitem[HIK22]{MR4484235}
Kei Hasegawa, Yusuke Isono, and Tomohiro Kanda, \emph{Note on bi-exactness for
  creation operators on {F}ock spaces}, J. Math. Soc. Japan \textbf{74} (2022),
  no.~3, 903--944. \MR{4484235}

\bibitem[Iso13]{MR3004955}
Yusuke Isono, \emph{Weak exactness for {$C^\ast$}-algebras and application to
  condition ({AO})}, J. Funct. Anal. \textbf{264} (2013), no.~4, 964--998.
  \MR{3004955}

\bibitem[Kir77]{MR0512362}
Eberhard Kirchberg, \emph{{$C\sp*$}-nuclearity implies {CPAP}}, Math. Nachr.
  \textbf{76} (1977), 203--212. \MR{512362}

\bibitem[Kir95a]{MR1403994}
\bysame, \emph{Exact {${\rm C}^*$}-algebras, tensor products, and the
  classification of purely infinite algebras}, Proceedings of the
  {I}nternational {C}ongress of {M}athematicians, {V}ol. 1, 2 ({Z}\"{u}rich,
  1994), Birkh\"{a}user, Basel, 1995, pp.~943--954. \MR{1403994}

\bibitem[Kir95b]{MR1322641}
\bysame, \emph{On subalgebras of the {CAR}-algebra}, J. Funct. Anal.
  \textbf{129} (1995), no.~1, 35--63. \MR{1322641}

\bibitem[Mag98]{MR1616512}
Bojan Magajna, \emph{A topology for operator modules over {$W^*$}-algebras}, J.
  Funct. Anal. \textbf{154} (1998), no.~1, 17--41. \MR{1616512}

\bibitem[Mag00]{MR1750836}
\bysame, \emph{{$C^*$}-convex sets and completely bounded bimodule
  homomorphisms}, Proc. Roy. Soc. Edinburgh Sect. A \textbf{130} (2000), no.~2,
  375--387. \MR{1750836}

\bibitem[OP04]{MR2052608}
Narutaka Ozawa and Sorin Popa, \emph{Some prime factorization results for type
  {${\rm II}_1$} factors}, Invent. Math. \textbf{156} (2004), no.~2, 223--234.
  \MR{2052608}

\bibitem[Oza04]{MR2079600}
Narutaka Ozawa, \emph{Solid von {N}eumann algebras}, Acta Math. \textbf{192}
  (2004), no.~1, 111--117. \MR{2079600}

\bibitem[Oza06]{MR2211141}
\bysame, \emph{A {K}urosh-type theorem for type {$\rm II_1$} factors}, Int.
  Math. Res. Not. (2006), Art. ID 97560, 21. \MR{2211141}

\bibitem[Oza07]{MR2370001}
\bysame, \emph{Weakly exact von {N}eumann algebras}, J. Math. Soc. Japan
  \textbf{59} (2007), no.~4, 985--991. \MR{2370001}

\bibitem[Oza10]{MR2730894}
\bysame, \emph{A comment on free group factors}, Noncommutative harmonic
  analysis with applications to probability {II}, Banach Center Publ., vol.~89,
  Polish Acad. Sci. Inst. Math., Warsaw, 2010, pp.~241--245. \MR{2730894}

\bibitem[Pim97]{MR1426840}
Michael~V. Pimsner, \emph{A class of {$C^*$}-algebras generalizing both
  {C}untz-{K}rieger algebras and crossed products by {${\bf Z}$}}, Free
  probability theory ({W}aterloo, {ON}, 1995), Fields Inst. Commun., vol.~12,
  Amer. Math. Soc., Providence, RI, 1997, pp.~189--212. \MR{1426840}

\bibitem[Pop93]{MR1198815}
Sorin Popa, \emph{Markov traces on universal {J}ones algebras and subfactors of
  finite index}, Invent. Math. \textbf{111} (1993), no.~2, 375--405.
  \MR{1198815}

\bibitem[PV14]{MR3259044}
Sorin Popa and Stefaan Vaes, \emph{Unique {C}artan decomposition for {$\rm
  II_{1}$} factors arising from arbitrary actions of hyperbolic groups}, J.
  Reine Angew. Math. \textbf{694} (2014), 215--239. \MR{3259044}

\bibitem[Sau83]{MR0703809}
Jean-Luc Sauvageot, \emph{Sur le produit tensoriel relatif d'espaces de
  {H}ilbert}, J. Operator Theory \textbf{9} (1983), no.~2, 237--252.
  \MR{703809}

\bibitem[Tak72]{MR0303307}
Masamichi Takesaki, \emph{Conditional expectations in von {N}eumann algebras},
  J. Functional Analysis \textbf{9} (1972), 306--321. \MR{303307}

\bibitem[Ued99]{MR1738186}
Yoshimichi Ueda, \emph{Amalgamated free product over {C}artan subalgebra},
  Pacific J. Math. \textbf{191} (1999), no.~2, 359--392. \MR{1738186}

\bibitem[Ued05]{MR2152505}
\bysame, \emph{H{NN} extensions of von {N}eumann algebras}, J. Funct. Anal.
  \textbf{225} (2005), no.~2, 383--426. \MR{2152505}

\bibitem[Ued08]{MR2546003}
\bysame, \emph{Remarks on {HNN} extensions in operator algebras}, Illinois J.
  Math. \textbf{52} (2008), no.~3, 705--725. \MR{2546003}

\bibitem[VDN92]{MR1217253}
D.~V. Voiculescu, K.~J. Dykema, and A.~Nica, \emph{Free random variables}, CRM
  Monograph Series, vol.~1, American Mathematical Society, Providence, RI,
  1992, A noncommutative probability approach to free products with
  applications to random matrices, operator algebras and harmonic analysis on
  free groups. \MR{1217253}

\end{thebibliography}
\nocite{MR1426840}

\end{document}